\documentclass[12pt]{article}

\usepackage{amsfonts}
\usepackage{amsmath}
\usepackage{amssymb}
\usepackage{latexsym}
\usepackage{amscd}
\usepackage{amsthm}

\newtheorem{lemma}{Lemma}

\newtheorem{theorem}{Theorem}
\newtheorem{cor}{Corollary}

\title{Growth of matrix products and mixing properties of the horocycle
flow}
\author{F\"edor Nazarov and Ekaterina Shulman \footnote{Research of the
second author supported by the Israel Science Foundation of the
Israel Academy of Sciences and Humanities}}
\begin{document}
\date{}
%\begin{abstract}
%\end{abstract}

\maketitle

\section{Introduction}

In this paper we investigate the following problem. Let $H(t) =
\begin{pmatrix} 1 & t \\  0 & 1 \end{pmatrix}$ and let
$\Phi_*=\{\Phi_n\}$ be an arbitrary sequence of matrices from
$SL(2,\mathbb{R})$. We will consider the sequence of products
$P_n(t) = \Phi_n H(t)\Phi_{n-1} H(t) \, ... \, \Phi_1 H(t)$ and
denote by $\mathfrak B(\Phi_*)$ the set of those periods $t \in
\mathbb R_+$ for which the sequence $\{P_n(t)\}$ is bounded:
\[
\mathfrak B(\Phi_*) = \left\{t\in\mathbb R_+\colon \sup_{n\ge 1}
||P_n(t)||<\infty \right\}\,.
\]
The question is: {\it how large the set $\mathfrak B(\Phi_*)$ can
be?} We present three results on this subject. The first one shows
that for every $\{\Phi_n\}$ the set $\mathfrak B(\Phi_*)$ is  not
``very large":
\begin{theorem}\label{main}
For every sequence $\Phi_*$, the set $ \mathfrak B (\Phi_*)$ has
finite measure.
\end{theorem}
It should be noted that for sequences $\Phi_*$ of some special
types this was already established in \cite{P-R}. Our main
innovation, which gives us the possibility to handle the general
case, is using of potential theory (Lemma \ref{harmon}).

The next two results demonstrate that the conclusion of Theorem
\ref{main} cannot be strengthened too much. Namely,  Theorem
\ref{seq} (section 5) shows that the exceptional set $\mathfrak
B(\Phi_*)$ can contain an arbitrary given sequence. In Theorem
\ref{EUS} (section 6) we produce an example of a sequence $\Phi_*$
for which the set $\mathfrak B(\Phi_*)$ is essentially unbounded,
that is $|\mathfrak B(\Phi_*) \bigcap [a,+\infty)| > 0$ for all
$a>0$ (We denote by $|E|$ the Lebesgue measure of a set $E\subset
\mathbb{R}$).

\paragraph{Motivation: stable quasi-mixing of the horocycle flow.}

In a recent work of L. Polterovich and Z. Rudnick \cite{P-R} the
authors considered the behavior of one-parameter subgroup of a Lie
group under the influence of a sequence of kicks. We remind some
basic concepts of their paper.

Let a Lie group $G$ act on a set $X$, and $\big( h^t \big)_{t\in
\mathbb R}$ be a one-parameter subgroup of $G$; we consider it as
a dynamical system acting on $X$ with continuous time $t$. We
perturb this system by a sequence of kicks $\{\phi_i\}\subset G$.
The kicks arrive periodically in time with some positive period
$t$. The dynamics of the kicked system is described by a sequence
of products $P_i(t) = \phi_ih^{t} \phi_{i-1} h^t \, ... \, \phi_1
h^t $ that depend on the period $t$. We treat $t$ as a parameter
and $i$ as a discrete time. Then the trajectory of a point $x\in
X$ is defined as $x_i=P_i(t)x$.

A dynamical property of a subgroup $(h^t)$ is called {\it kick
stable}, if for every sequence of kicks $\{\phi_i\}$, the kicked
sequence $P_i(t)$ inherits this property for a ``large" set of
periods $t$. The property we  will concentrate on in this paper,
is quasi-mixing.

A sequence $\{P_i\}$ acting on a compact measure space $(X, \mu)$
by measure-preserving automorphisms is called {\it mixing} if for
any two $L_2$-functions $F_1$ and $F_2$ on $X$
$$
\int_X F_1(P_i x)F_2(x)d\mu \rightarrow \int_X F_1(x)d\mu \int_X
F_2(x)d\mu
$$
when $i\rightarrow \infty$. A sequence $\{P_i\}$ is called {\it
quasi-mixing} if there exists a subsequence $\{i_k\}\rightarrow
\infty$ such that for any two $L_2$-functions $F_1$ and $F_2$ on
$X$
$$
\int_X F_1(P_{i_k}x)F_2(x)d\mu \rightarrow \int_X F_1(x)d\mu
\int_X F_2(x)d\mu
$$
when $k \rightarrow \infty$.

\medskip
In what follows, $G = PSL(2,\mathbb{R})$, $\Gamma \subset
PSL(2,\mathbb{R})$ is a lattice, that is a discrete subgroup such
that the Haar measure of the quotient space
$X=PSL(2,\mathbb{R})/\Gamma$ is finite. The group
$PSL(2,\mathbb{R})$ acts on $X$ by left multiplication. This
action evidently preserves the Haar measure. The principal tool
used in \cite{P-R} for the study of stable mixing in this setting,
is the Howe-Moore theorem which gives the geometric description of
mixing systems: if the sequence $\{P_i\}$ tends to
infinity\footnote{i.e., for every compact subset $Q\subset G$ the
sequence $\{P_i\}$ eventually leaves $G$} then it is mixing. It
was also shown that the converse is true. In a similar way, the
quasi-mixing is equivalent to the unboundedness of the sequence
$\{P_i\}$.

It follows from the Howe-Moore theorem that the horocycle flow
$$
H(t) =
\begin{pmatrix}
  1 & t \\
  0 & 1
\end{pmatrix}
$$
on $PSL(2,\mathbb{R})/\Gamma$ is mixing. An example given in
\cite[Remark~3.3.E]{P-R} shows that this flow is not stably
mixing. Our Theorem \ref{main} says that {\it it is stably
quasi-mixing}. This answers the question raised by Polterovich and
Rudnick \cite[Question~3.3.B]{P-R}.

\medskip\par\noindent

Let us mention a corollary to Theorem \ref{main} that pertains to
second order difference equations. It was shown in \cite{P-R} that
for a kick sequence of the form $\begin{pmatrix} 1 & 0 \\ c_n & 1
\end{pmatrix}$, the unboundedness of the evolution is equivalent to
the existence of unbounded solutions for the discrete
Schr\"odinger-type equation
\begin{equation}\label {Shr}
q_{k+1} - (2+tc_k)q_k + q_{k-1}=0,  \qquad  k\geq 1.
\end{equation}
So our result implies
\begin{cor} For every sequence $\{c_n\}$, the set of the parameters
$t \in \mathbb R_+$ for which all solutions of the difference
equation (\ref{Shr}) are bounded, has finite measure.
\end{cor}

\medskip\par\noindent\textbf{Acknowledgment}
We are grateful to L.~Polterovich, M.~Sodin, D.~Burago and P.~Yuditskii
for involving us into this subject and for very helpful discussions.

\section{Outline of the proof of Theorem \ref{main}}

Our proof of Theorem \ref{main} consists of several steps and uses
some preliminary results (Lemmas 1-5 below).  For convenience of a
reader we begin with an outline of this proof.

\medskip\par\noindent{\bf Step 1.} First of all we show, that the
problem can be reduced to the case of bounded sequences of kicks
$\Phi_*$ (Lemma~\ref{ogr}).

\medskip\par\noindent{\bf Step 2.} For bounded sequences we use the
Iwasawa's decomposition of $2\times 2$ matrices:
$$
\Phi_n=\begin{pmatrix} 1 & s_n \\ 0 & 1\end{pmatrix}
\begin{pmatrix} \lambda_n & 0 \\ 0 & \frac{1}{\lambda_n}\end{pmatrix}
\begin{pmatrix} \cos \alpha_n & -\sin \alpha_n \\ \sin \alpha_n  & \cos
\alpha_n\end{pmatrix}:= H(s_n)D(\lambda_n)R(\alpha_n).
$$
with bounded sequences $\{H(s_n)\}$ and $\{D(\lambda_n)\}$ and
$-\frac{\pi}{2} \leq \alpha \leq \frac{\pi}{2}$. Denoting by
$$
q=\limsup_{n\rightarrow\infty}\frac{1}{n}\sum_{j=1}^n |\alpha_j|
$$
we then consider two cases separately: $q=0$ and $q>0$.

\medskip\par\noindent{\bf Step 3.} In the case of ``small" angles
($q=0$), the sequence $\{\Phi_n\}$ is ``close" (in  some sense
which we define below) to the bounded sequence
$\{H(s_n)D(\lambda_n)\}$ of upper-triangular matrices. This
implies that the set $ \mathfrak B (\Phi_*)$ is bounded (see
Lemma~\ref{vozm}).

\medskip\par\noindent{\bf Step 4.} In the case $q>0$ we extend our
problem to the complex plane and consider $SL(2,\mathbb C) $-matrices
$H(z)=\begin{pmatrix} 1 & z \\ 0 & 1\end{pmatrix}$. Respectively,
$P_n(z)=\Phi_n H(z)\cdot\,...\,\cdot\Phi_1 H(z)$. We show that the
set
$$
E=\{z\in\mathbb{C}\colon \limsup_{n\rightarrow\infty}
\frac{\log\|P_n(z)\|}{n}=0 \}
$$
is contained in $\mathbb{R}$ and has finite length. So we not only
prove that the sequence $\{P_n\}$ is unbounded  but prove that it
has exponential growth for all $t$ apart of a set of finite
measure.

\smallskip
In order to show that
\begin{equation}
\limsup_{n\rightarrow\infty} \frac{\log\|P_n(z)\|}{n}>0  \qquad z
\in \mathbb{C}\setminus\mathbb{R}.
\end{equation}
we have to estimate $\|P_n(z)\|$ from below. To this aim we
use the quadratic form
$$
Q(x) = {\rm Im} \left( x_1\bar{x}_2 \right), \qquad x=(x_1,x_2)
\in \mathbb{C}^2,
$$
which has the following properties:
\begin{itemize}
\item[(i)] for arbitrary  $y$,  $\|y\|^2 \geq 2Q(y)$,
\item[(ii)] for every $z \in \mathbb{C}$ with ${\rm Im} z > 0$,
one has
\begin{equation}
Q \left(H(z)\Phi_n H(z)x\right)\geq Q(x)\left( 1 +
\frac{|\alpha_n|{\rm Im} z}{2k(1+ |z|)} \right).
\end{equation}
\end{itemize}
Due to these properties we get that
\begin{equation*}
\overline{\lim_{n\rightarrow\infty}}\ \ \frac{\log \|P_n(z)\|}{n}=
\overline{\lim_{n\rightarrow\infty}}\ \ \frac{1}{n}\log
\Big\|\prod_ {1\leq j \leq n}^{\curvearrowleft}H(z/2)\Phi_j
H(z/2)\Big\| > 0.
\end{equation*}
(We denote by $\displaystyle \prod_{1\leq j \leq
n}^{\curvearrowleft}A_j$ the matrix product $A_nA_{n-1}\ \ldots \
A_1$.)

The claim that $\left|E\right|<\infty$ follows now from a
potential theory lemma (Lemma~\ref{harmon}) applied to the
subharmonic functions
$$
u_n(z)=\frac{\log \|P_n(z)\|}{n}.
$$

\section{Preliminaries}
\begin{lemma}\label{ogr}
If the sequence of kicks $\Phi _*$ is unbounded then the set
$\mathfrak B(\Phi_*)$ is empty.
\end{lemma}

\begin{proof}
Aiming at a contradiction, we assume that for some $t>0$ the sequence
$\{\|P_n(t)\|\}$ is bounded by $M$. Taking into account that
$\|A^{-1}\| = \|A\|$ for $A\in SL(2,\mathbb{R})$, we obtain
\begin{multline*}
\|\Phi_n\| = \|P_n(t) \left( P_{n-1}(t)\right)^{-1}
\left(H(t)\right)^{-1}\| \\
\leq  \|P_n(t) \| \cdot \|\left( P_{n-1}(t) \right)^{-1}\| \cdot
\|\left(H(t)\right)^{-1}\|\leq M^2\|H(t)\|
\end{multline*}
which contradicts to the unboundedness of $\{\Phi _n\}$.
\end{proof}

Thus, in what follows, we assume that the sequence $\{\Phi _n\}$
is bounded.

\begin{lemma}\label{rost} Let $\Psi_*$ be a bounded sequence of
upper-triangular matrices:
\begin{equation}\label{d}
\Psi_n = \begin{pmatrix} \lambda_n & s_n \\ 0 &
\frac{1}{\lambda_n}\end{pmatrix}
\end{equation}
and
\begin{equation}\label{max}
t_0=\max\{|s_n/\lambda_n|, n \in \mathbb{N}\}.
\end{equation}
Then, for all $t>t_0$ and all $K>0$, there exists $N \in
\mathbb{N}$ such that, for each $j \in \mathbb{N}$, at least one
of $N$ products
\begin{equation} \label{piece}
\Psi_{j+m}H(t) \cdot \, ... \, \cdot \Psi_{j+1}H(t) , \qquad m=1,
2, \ldots, N
\end{equation}
has norm larger than $K$.
\end{lemma}

\begin{proof}
Let us fix $t>t_0$. It follows by induction in $m$ that
$$
\Psi_{j+m}H(t) \cdot \, ...\, \cdot \Psi_{j+1}H(t) =
\begin{pmatrix} \Pi_{j,m} & t\Pi_{j,m} S_{j,m}(t) \\ 0 &
\Pi_{j,m}^{-1}\end{pmatrix}
$$
where
\begin{equation}\label{Pi}
\Pi_{j,m} = \lambda_{j+1} \cdot\, ...\, \cdot \lambda_{j+m},
\end{equation}
\begin{equation}\label{Es}
S_{j,m}(t) = \left( t + \frac {s_{j+1}}{\lambda_{j+1}} \right) +
\frac{t + \frac {s_{j+2}}{\lambda_{j+2}}}{\lambda_{j+1}^2} + \,
... \, +\frac{t + \frac {s_{j+m}}{\lambda_{j+m}}}{\lambda_{j+1}^2
\cdot\, ... \,\cdot \lambda_{j+m-1}^2}.
\end{equation}
Suppose that the assertion of the lemma is wrong. Then there
exists $K>0$ such that for any $N$ one can find $j$ with the
property that all products (\ref{piece}) have norm less than $K$.
It follows that $K^{-1} < |\Pi_{j,m}| < K$ and $|S_{j,m}(t)| <
\frac{K^2}{t}.$

The denominators of the summands in the right hand side of (\ref{Es})
are equal to $\Pi_{j,k}^2$, so they do not exceed $K^2$. On the
other hand, $\displaystyle t + \frac {s_{j+1}}{\lambda_{j+1}} >
t-t_0$. Hence
$$
|S_{j,m}(t)| > m \frac{t-t_0}{K^2}.
$$
Thus $\displaystyle m \frac{t-t_0}{K^2}< \frac{K^2}{t}$. In
particular, this is true for $m = N$. But the inequality
$\displaystyle N \frac{t-t_0}{K^2}< \frac{K^2}{t}$ can not hold
for all $N$. We obtained a contradiction.
\end{proof}

Let us say that a sequence $\{a_n\}$ of complex numbers satisfies
the condition $(*)$ if
\[
\forall \varepsilon > 0 \quad {\rm and}\quad  \forall N \in
\mathbb{N} \quad \exists i \in \mathbb{N} \quad {\rm such\ that }
\quad \max_{1\le j \le N }|a_{i+j}| < \varepsilon\,. \eqno (*)
\]

\begin{lemma}\label{vozm}
Let $\Psi_*$ be a bounded sequence of upper-triangular matrices.
If a sequence of matrices $\Phi_*$ is so close to $\Psi_*$ that
$\{\|\Phi_n-\Psi_n\|\}$ satisfies $(*)$, then $\mathfrak
B(\Phi_*)\subseteq [0;t_0]$  with $t_0$ the same as in
\eqref{max}.
\end{lemma}
\begin{proof}
Assume, to the contrary, that for some $t>t_0$ there exists $M>0$
such that $\|P_n(t)\| \le M$ for all $n \in \mathbb{N}$. Applying
Lemma \ref{rost} to the sequence $\Psi_*$ with $K =2M^2$, we
obtain a positive integer $N$ such that, for any $j$, at least one
of the products (\ref{piece}) with $m\le N$ has norm larger than
$2M^2$.

Fix arbitrary $\delta>0$ and $C>1 + \sup\|\Phi_n\|$ and choose
$\varepsilon > 0$ with
$$
\varepsilon < \frac{\delta (C-1)}{(C^N-1)\|H(t)\|^N}\;.
$$
For these $N$ and $\varepsilon$, find $i$ according to condition $(*)$:
$$
\|\Phi_{i+j} - \Psi_{i+j}\| < \varepsilon  \qquad j=1,2, \ldots,
N.
$$
By our choice of $N$, there exists $m$, $1 \leq m \leq N$, for
which
$$
\|\Psi_{i+m}H(t)\cdot \, ... \, \cdot\Psi_{i+1}H(t) \| > K\,.
$$
Now, we estimate the product $ \Phi_{i+m}H(t)\cdot \, ... \,
\cdot\Phi_{i+1}H(t)$. On the one hand, it is close to
$\Psi_{i+m}H(t)\cdot \,...\, \cdot \Psi_{i+1}H(t)$:
\begin{multline*}
\|\Phi_{i+m}H(t)\cdot\,...\, \cdot\Phi_{i+1}H(t)  - \Psi_{i+m}H(t)
\cdot\,...\,\cdot\Psi_{i+1}H(t)\|
\\ \\
\le
\|\Phi_{i+m}H(t)\cdot\,...\,\cdot\Phi_{i+2}H(t)(\Phi_{i+1}-\Psi_{i+1})H(t)\|\,
\\ \\
+ \, \|\Phi_{i+m}H(t)  \cdot\,...\,\cdot
\Phi_{i+2}H(t)\Psi_{i+1}H(t)-\Psi_{i+m}H(t)\cdot\,...\,\cdot\Psi_{i+1}H(t)\|
\\ \\
\le \, ...\, \le
\|\Phi_{i+m}H(t)\cdot\,...\,\cdot\Phi_{i+2}H(t)(\Phi_{i+1}-\Psi_{i+1})H(t)\|\,
\\ \\
+ \, \|\Phi_{i+m}H(t)
\cdot\,...\,\cdot(\Phi_{i+2}-\Psi_{i+2})H(t)\Psi_{i+1}H(t)\|
\\ \\
+\, ...\, + \|(\Phi_{i+m}-\Psi_{i+m})H(t)\Psi_{i+m-1}H(t) \cdot \,
...\, \cdot\Psi_{i+1}H(t)\|
\\ \\
\leq \varepsilon  (C^{m-1} + C^{m-2} + \,...\, + 1)\|H(t)\|^m \leq
\frac{\varepsilon (C^m-1)\|H(t)\|^m}{C-1} < \delta.
\end{multline*}
Therefore
\begin{equation}\label{larger}
\|\Phi_{i+m}H(t)\cdot\,...\,\cdot\Phi_{i+1}H(t)\| \geq
\|\Psi_{i+m}H(t) \cdot\,...\,\cdot\Psi_{i+1}H(t)\| -
\delta > 2M^2 - \delta .
\end{equation}

On the other hand,
\begin{multline*}
\|\Phi_{i+m}H(t) \cdot\,...\,\cdot \Phi_{i+1}H(t)\| =
\|P_{i+m}(t)\left( P_{i}(t)\right)^{-1}\|
\\ \\
\leq  \|P_{i+m}(t)\|\cdot\|\left( P_{i}(t)\right)^{-1}\| =
\|P_{i+m}(t)\|\cdot\|P_i(t)\| \leq M^2.
\end{multline*}
which contradicts (\ref{larger}).
\end{proof}

We will also need two auxiliary results from the classical
potential theory in the spirit of Wiener's criterion \cite{L}. Let
$\Omega$ be a bounded domain in the complex plane, $z_0\in
\Omega$. Recall that the harmonic measure $\omega$ on the boundary
$\partial \Omega$ respect to a point $z_0$ is defined by the condition
\begin{equation}\label{harmeas}
u(z_0)= \int_{\partial \Omega}u(z)d\omega ,
\end{equation}
for all harmonic continuous function on $\overline{\Omega}$, and
that for subharmonic $u$ the equality should be changed by the
inequality
\begin{equation}\label{subharmeas}
u(z_0)\le \int_{\partial \Omega}u(z)d\omega.
\end{equation}
We will denote by $|E|$ the Lebesgue measure of a set $E\subset
\mathbb{R}$.

\begin{lemma}\label{set}
Let $E$ be a closed subset of $[1,+\infty]$ of infinite length
with the property that
\begin{equation}\label{inters}
|E\cap [a,4a]| \le 1\;\;\; for\;each\;a>0.
\end{equation}
Fix $z_0\in \mathbb{C}\setminus E$.  Take $R > |z_0|$ and consider
the domain
$$\Omega_R = \{z\in \mathbb{C}: |z|<2R, z\notin E_{1,R}\},$$
where by $E_{a,b}$ we denote $E\cap [a,b]$. Then the harmonic
measure $\omega_R$  on  $\partial\Omega_R$, associated with the
point $z_0$, satisfies the condition
$$
\lim_{R\to\infty}\omega_R(T_R)\log(1+2R) = 0
$$
where $T_R = \partial \Omega_R \cap \{|z|=2R\}$.

\end{lemma}
\begin{proof}
Observe, first of all, that (\ref{inters}) implies that for each
$a>0$,
$$\int_{E_{a,\infty}}\frac{dt}{t} \le
\sum_{k=0}^{\infty}\frac{1}{4^ka}|E_{4^ka,4^{k+1}a}| \le
\frac{4}{3a}.
$$
Consider the auxiliary potential
$$U(z) = \int_{E_{1,R}}\log \left|1-\frac{z}{t}\right|dt.
$$
Notice that $U\in C(\overline{\Omega_R})$ \footnote{This follows,
for example, from the Continuity Principle (see \cite[section~
3.1] {Rans})} and $U$ is harmonic in $\Omega_R$. Since
$\partial\Omega_R$ consists of the circumference $T_R$ of radius
$2R$ centered at $0$ and the set $E_{1,R}$, we have
\begin{equation}\label{harmonic}
U(z_0)= \int_{T_R}U(z)d\omega_R(z) +
\int_{E_{1,R}}U(z)d\omega_R(z).
\end{equation}
Hence
\begin{equation}\label{ineq}
\int_{T_R}U(z)d\omega_R(z) \le |U(z_0)| -
\int_{E_{1,R}}U(z)d\omega_R(z).
\end{equation}
It follows from the definition of $U$ that
\begin{equation}\label{N-1}
|U(z_0)|\leq \int_{E_{1,R}}\log \left\{1+
\frac{|z_0|}{t}\right\}dt \leq |z_0|\int_{E_{1,R}}\frac{dt}{t}
\leq \frac{4}{3}|z_0|.
\end{equation}
Choose $b>0$ so that $|E_{1,b}| = 1$. We may suppose that $R
> 2+b+|z_0|$.

We claim that for every $z\in E_{1,R}$,
\begin{equation*}\aligned
U(z) \geq \int_{E_{1,b}}\log \left|1- \frac{z}{t}\right|dt & +
\int_{E_{\frac{1}{2}z, 2z}}\log \left|1- \frac{z}{t}\right|dt + \\
& \int_{E_{2z,+\infty}}\log \left|1- \frac{z}{t}\right| dt =
\mathcal{J}_1+\mathcal{J}_2+ \mathcal{J}_3
\endaligned
\end{equation*}
The reason for this inequality is that $\mathcal{J}_2$ and
$\mathcal{J}_3$ together give exactly the integral of {\em all}
negative values of the function $t \rightarrow \log \left| 1-
\frac{z}{t}\right|$ on $E_{1, +\infty}$, so the extension of the
upper limit from $R$ to $+\infty$ in $\mathcal{J}_3$ and possible
overlapping with $\mathcal{J}_1$ are not problems: essentially
what is said is that the integral of a real-valued function over a
set $F$ is not less than its integral over any subset of $F$ plus
the integral of all its negative values over any superset of $F$.

Observe that, since $\log|1-x|\geq-2x$ for $0 < x < 1/2$, we have
$$
\mathcal{J}_3 \geq -2z \int_{E_{2z, +\infty}} \frac{dt}{t} \geq
-\frac{4}{3}
$$
regardless of $z$. We also have (recall that $|E_{1, b}|=1$)
$$
\mathcal{J}_1 - \log z = \int_{E_{1,b}}\log \left|\frac{1}{z}-
\frac{1}{t}\right|\ dt \rightarrow - \int_{E_{1,b}}\log t\ dt
\qquad as \ \ z \rightarrow +\infty.
$$
Since $\mathcal{J}_1-\log z$ is a continuous function on $[1,
+\infty)$, we get $\mathcal{J}_1 \geq \log z -C_1$ where $C_1$ is
some large constant independent on $R$. Next,
\begin{equation*}\aligned
\mathcal{J}_2 = \int_{E_{\frac{1}{2}z, 2z}}& \log |t- z|\ dt  -
\int_{E_{\frac{1}{2}z, 2z}}\log t \ dt \geq \\
& \int_{[z-1, z+1]} \log |t-z|\ dt - \log(2z)|E_{\frac{1}{2}z,
2z}| \geq -2 -\log z - \log 2
\endaligned
\end{equation*}
(here, we used the inequality $|E_{\frac{1}{2}z, 2z}| \leq 1$).
Therefore, $U(z) \geq -C_1 -10/3 - \log 2 = -C_2$ on $E_{1, R}$
and, taking into account that $\omega_R(E_{1, R}) \leq
\omega_R(\partial\Omega)=1$, we get
\begin{equation}\label{N-2}
\int_{E_{1,R}} U(z)\ d\omega_R(z) \geq -C_2\omega_R(E_{1,R}) \geq
-C_2.
\end{equation}
Now by (\ref{ineq}),
\begin{equation}\label{ineq1}
\int_{T_R}U(z)d\omega_R(z) \le C_3,
\end{equation}
where $C_3=C_2+\frac{4}{3}|z_0|$.

The last observation we need is that for $z \in T_R$,
\begin{equation}\label{N-3}
U(z)\geq \int_{E_{1,R}} \log \left| 1- \frac{2R}{t}\right|dt.
\end{equation}
So (\ref{ineq1}) gives
$$
\omega_R(T_R)\leq C_3 \left\{\int_{E_{1,R}} \log \left| 1-
\frac{2R}{t}\right|dt \right\}^{-1}.
$$
Thereby,
\begin{multline*}
\omega_R(T_R)\log(1+2R) \leq C_3 \left\{ \frac1{\log (1+2R)}
\int_{E_{1,R}} \left| 1- \frac{2R}{t}\right|\, dt \right\}^{-1}\\
= \left\{\int_{E_{1,R}}L_R(t)dt \right\}^{-1}.
\end{multline*}
Note that $0 \leq L_R(t) \leq 1$ for each $t \in E_{1,R} $ and
$L_R(t)\rightarrow 1$ as $R\rightarrow +\infty $ for every fixed
$t \in E$. Thereby, $\int_{E_{1,R}}L_R(t)dt \rightarrow |E| =
+\infty$,
and we are done.
\end{proof}

\begin{lemma}\label{harmon}
Let $u_n(z)$ be a sequence of continuous subharmonic functions
satisfying the estimate $u_n(z) \leq \log (1+|z|) + A$ for all
$n\geq1$, $z \in \mathbb{C}$ and some $A>0$. If
$E\subset\mathbb{R}$ has infinite length and
$\limsup_{n\rightarrow\infty}u_n(z) \leq 0$ for all $z \in E$,
then $\limsup_{n\rightarrow\infty}u_n(z) \leq 0$ for all $z \in
\mathbb{C}$.
\end{lemma}

\begin{proof}
Since every measurable set of infinite length contains a
closed subset of infinite length we may assume without loss of
generality that $E$ is closed. Also we may assume that $|E\cap
[1,+\infty)| = +\infty$. Indeed, otherwise $E\cap[-\infty,-1]| =
+\infty $ and we may consider the set $-E$ and functions $u_n(-z)$
instead. Thus we can always assume that $E \subset [1,+\infty]$.
The last regularization we need is the following. Take any dyadic
interval $I_k = [2^{k-1},2^k] (k = 1,2,...)$. If $|E\cap I_k| <
1/3$, leave the corresponding piece of $E$ alone. Otherwise
replace it by some subset of length exactly $1/3$. The resulting
set still has infinite length. Indeed if we made finitely many
replacements, we dropped only a set of finite length from $E$, and
if we made infinitely many replacements, we have infinitely many
disjoint pieces of length  $1/3$  in the resulting set. After such
regularization, the set $E$ enjoys the property (\ref{inters}). We
will use the notation introduced in the previous Lemma.

Choose $z_0 \in \mathbb{C}$; we have to prove that
$\limsup_{n\to\infty}u_n(z_0) \le 0$. This is evident if $z_0\in
E$, so we assume that $z_0\in \mathbb{C}\setminus E$.

For $R> |z_0|$, we have, by (\ref{subharmeas}),
$$u_n(z_0)\le \int_{\partial\Omega_R}u_n(z)d\omega_R(z) =
\int_{T_R}u_n(z)d\omega_R(z) + \int_{E_{1,R}}u_n(z)d\omega_R(z).$$
Note that, for fixed $R$, the length of $E_{1,R}$ is finite, $u_n$
are uniformly bounded from above on $E_{1,R}$, and $\limsup_{n\to
\infty}u_n(z)\le 0$ for all $z\in E_{1,R}$. Therefore, the Fatou
lemma yields
$$\limsup_{n\to\infty}\int_{E_{1,R}}u_n(z)d\omega_R(z)\le 0$$
and thereby
$$\limsup_{n\to\infty}u_n(z_0) \le \sup_{n\ge
1}\int_{T_R}u_n(z)d\omega_R(z)\le \omega_R(T_R)(\log(1+2R)+A)$$
for any fixed $R$. By Lemma~\ref{set}, the result follows.
\end{proof}

\section{The proof of Theorem \ref{main}}

Now
we can prove Theorem \ref{main}.

\begin{proof}
We use Iwasawa's decomposition of $2\times 2$ matrices:
\begin{multline*}
\Phi_n=\begin{pmatrix} a_n & b_n \\ c_n & d_n\end{pmatrix}=
\begin{pmatrix} 1 & s_n \\ 0 & 1\end{pmatrix}
\begin{pmatrix} \lambda_n & 0 \\ 0 & \frac{1}{\lambda_n}\end{pmatrix}
\begin{pmatrix} \cos \alpha_n & -\sin \alpha_n \\ \sin \alpha_n  & \cos
\alpha_n\end{pmatrix} \\
= H(s_n)D(\lambda_n) R(\alpha_n),
\end{multline*}
where $\displaystyle \alpha_n = \arcsin\frac{c_n {\rm
sign}(d_n)}{\sqrt{c_n^2 + d_n^2}} \in \left[-\frac{\pi}{2}, \frac{\pi}{2} \right]$,
$\displaystyle s_n = \frac{a_n c_n+b_n d_n}{c_n^2 + d_n^2}$, $\displaystyle \lambda_n =
\frac{{\rm sign}(d_n)}{\sqrt{c_n^2 + d_n^2}}$.

By Lemma~\ref{ogr}, we may assume that the system $\Phi_*$ is
bounded. It follows that  both sequences $\{H(s_n)\}$ and
$\{D(\lambda_n)\}$ are bounded. Indeed,
$$
s_n = \frac{(a_n, b_n)\cdot(c_n, d_n)}{\|(c_n, d_n)\|^2} \leq
\frac {\|(a_n, b_n)\|}{\|(c_n, d_n)\|} = \frac {\|(a_n,
b_n)\|}{\|(d_n, -c_n)\|} \leq  \frac {\|(a_n, b_n)\|}{1/\|(a_n,
b_n)\|} \leq C^2
$$
where $C > 1+ \sup \{\|\Phi_n\|\}$. Thus $\{H(s_n)\}$ is a bounded
sequence. The boundedness of $\{D(\lambda_n)\}$ follows from the
equality
$$
D(\lambda_n) = H(s_n)^{-1}\Phi_n R(\alpha_n)^{-1}.
$$

Now denote
$$
q=\limsup_{n\rightarrow\infty}\frac{1}{n}\sum_{j=1}^n |\alpha_j|
$$
and consider two cases separately: $q=0$ and $q>0$. In both cases
we will prove stronger statements than the assertion of
Theorem~\ref{main}.

\medskip\par\noindent\textbf{Case A: $q=0$.} We will show that in
this case the set $\mathfrak B(\Phi_*)$ is bounded. Note, that in
this case the sequences $\{\alpha_n \}$ and hence $\{\sin \alpha_n
\}$ satisfy condition $(*)$.

Since
\begin{multline*}
\|\Phi_n - H(s_n)D(\lambda_n) \| \leq
\|H(s_n)\|\|D(\lambda_n)\|
\left|\left|\begin{pmatrix} \cos (\alpha_n) -1 & -\sin \alpha_n \\
\sin \alpha_n  & \cos(\alpha_n) -1
\end{pmatrix} \right|\right|
\end{multline*}
$$
\le 2|\sin \alpha_n|\|H(s_n)\|\|D(\lambda_n)\|,
$$
the sequence $\Phi_*$ is close in $(*)$-sense to the sequence of
upper-triangular matrices $\{H(s_n)D(\lambda_n)\}$, i.e, the
sequence of norms $ \|\Phi_n - H(s_n)D(\lambda_n) \|$ satisfies
$(*)$. Thus, according to Lemma~\ref{vozm}, the sequence of
evolutions $\{P_{n}(t)\}$ is unbounded for every $t>t_0$, where
$$
t_0=\max\{|s_n/\lambda_n^2|, n \in \mathbb{N}\}.
$$
This completes the proof in the case~A. \hfill$\Box$

\medskip\par\noindent\textbf{Case B: $q>0$.} Extending
our problem to the complex plane, we consider $SL(2,\mathbb
C)$-matrices $H(z)=\begin{pmatrix} 1 & z \\ 0 & 1\end{pmatrix}$
instead of $H(t)$. Respectively, $P_n(z)=\Phi_n H(z)...\Phi_1
H(z)$. We will show that the set
$$
E=\{z\in\mathbb{C}\colon \limsup_{n\rightarrow\infty}
\frac{\log\|P_{n}(z)\|}{n}=0 \}
$$
is contained in $\mathbb{R}$ and has finite length. So we not only
prove that the sequence $P_*$ is unbounded  but prove that it has
exponential growth for all $t$ apart of a set of finite measure.

Our first task is to show that
\begin{equation}\label{complex}
\limsup_{n\rightarrow\infty} \frac{\log\|P_n(z)\|}{n}>0,  \qquad z
\in \mathbb{C}\setminus\mathbb{R}.
\end{equation}
Assume that ${\rm Im} z >0$ (the case ${\rm Im} z <0$ can be
considered in a similar way). Let us consider the quadratic form
$$
Q(x) = {\rm Im} (x_1\bar{x}_2), \qquad  x=(x_1,x_2) \in
\mathbb{C}^2.
$$
For any matrix $A \in SL(2,\mathbb{R})$ and $x \in \mathbb{C}^2$,
one has $Q(Ax)=Q(x)$. On the other hand, for every $z\in
\mathbb{C}$ with ${\rm Im} z > 0$, one has
$$
Q(H(z)x) = Q (x) + {\rm Im} z |x_2|^2 \geq Q (x)\left( 1 + {\rm
Im} z \frac{|x_2|}{|x_1|} \right).
$$

Now, we need one more lemma.
\begin{lemma}\label{oml}
Let $\alpha_n \in [-\pi/2, \pi/2]$, then there exists $k \geq 1$
such that for all $n \in \mathbb{N}$
\begin{equation}\label{eq_21}
Q \left( H(z)H(s_n)D(\lambda_n)R(\alpha_n)H(z)x\right)\geq
Q(x)\left( 1 + \frac{|\alpha_n|{\rm Im} z}{2k(1+ |z|)} \right).
\end{equation}
\end{lemma}

\begin{proof}[Proof of Lemma \ref{oml}.] We split the proof into
two cases.

\medskip\par\noindent\textbf{Case 1:} $\displaystyle
\frac{|x_2|}{|x_1|}\geq \frac{|\alpha_n|}{2(1+|z|)}$ . Then

\begin{multline*}
Q \left( H(z)H(s_n)D(\lambda_n)R(\alpha_n)H(z)x\right)
\\
\geq Q \left(H(s_n)D(\lambda_n)R(\alpha_n)H(z)x\right) = Q \left(
H(z)x\right)
\\
\geq Q(x)\left( 1 + {\rm Im} z \cdot \frac{|x_2|}{|x_1|} \right)
\geq Q(x)\left( 1 + \frac{|\alpha_n|{\rm Im} z}{2(1+ |z|)}\right).
\end{multline*}

\medskip\par\noindent\textbf{Case 2:} Now, we suppose that
$\displaystyle \frac{|x_2|}{|x_1|}\leq
\frac{|\alpha_n|}{2(1+|z|)}$, and split the proof of the estimate
\eqref{eq_21} into $4$ steps.

\medskip\par\noindent 1. Let us estimate how the matrix $H(z)$
changes the ratio of coordinates. Since  $|\alpha_n| \leq \frac {\pi}{2}$
we have
\begin{multline*}
\left|\left[H(z)x \right]_2\right|=|x_2| \leq
\frac{|\alpha_n|}{2}|x_1|\frac{1}{1+|z|} \\
= \frac{|\alpha_n|}{2}\cdot|x_1|\cdot\left( 1- \frac{|z|}{1+|z|}
\right) \leq
\frac{|\alpha_n|}{2}\cdot|x_1|\cdot\left(1-\frac{|\alpha_n||z|}
{2(1+|z|)} \right)
\end{multline*}
where $[\ ]_2$ means the second coordinate. Next,
\begin{equation*} \left|\left[H(z)x \right]_1
\right| = \left|x_1 + z x_2 \right| \ge |x_1| \left|
1-|z|\frac{|x_2|}{|x_1|}\right|
> |x_1|\cdot\left(1-\frac{|\alpha_n||z|}{2(1+|z|)}
\right)
\end{equation*}
since $\displaystyle |z|\frac{|x_2|}{|x_1|}\leq \frac
{|\alpha_n|}{2}\frac{|z|}{1+|z|} < \frac {|\alpha_n|}{2} <1$. Thus
\begin{equation*}
\left|\left[H(z)x \right]_2\right| \leq \frac{|\alpha_n|}{2}
\left|\left[H(z)x \right]_1\right|.
\end{equation*}

\medskip\par\noindent 2. It is easy to check the following property
of an orthogonal matrix $R(\alpha)$: for any $x\in \mathbb{C}^2$
and $|\alpha|\leq \frac{\pi}{2}$, the inequality
$|x_2|\leq\frac{|\alpha|}{2}|x_1|$ implies
$|[R(\alpha)x]_2|\geq\frac{|\alpha|}{2}|[R(\alpha)x]_1|$.
Therefore
\begin{equation*}
\left|\left[R(\alpha_n)H(z)x \right]_2\right| \geq
\frac{|\alpha_n|}{2} \left|\left[R(\alpha_n)H(z)x
\right]_1\right|.
\end{equation*}

Denote temporarily  $R(\alpha_n)H(z)x=y$, then $\displaystyle
|y_2| \geq \frac{|\alpha_n|}{2}\cdot|y_1|$.

\medskip\par\noindent 3. We set $k=C^2$ with $C=\sup_{i} ||\Phi_i||$
and obtain
\begin{equation*}
\frac{\left| \left[D(\lambda_n)y \right]_2\right|}{\left|
\left[D(\lambda_n)y \right]_1\right|}= \frac
{\left|\frac{1}{\lambda_n} \right||y_2|}{|\lambda_n||y_1|} \geq
\frac {1}{\lambda_n^2}\cdot \frac {|\alpha_n|}{2} >
\frac{|\alpha_n|}{2k}.
\end{equation*}

\medskip\par\noindent 4. At last
\begin{multline*}
Q \left( H(z)H(s_n)D(\lambda_n)R(\alpha_n)H(z)x\right) =
Q\left(H(z)H(s_n)D(\lambda_n)y\right)\\
= Q\left(H(s_n)D(\lambda_n)y\right)+ {\rm Im} z
\left|\left[H(s_n)D(\lambda_n)y \right]_2\right|^2=
Q\left(D(\lambda_n)y\right)\\
+ {\rm Im} z \left|\left[D(\lambda_n)y \right]_2\right|^2\geq Q
\left(D(\lambda_n)y\right)\left(1+ {\rm Im} z \cdot \frac{\left|
\left[D(\lambda_n)y \right]_2\right|}{\left|
\left[D(\lambda_n)y \right]_1\right|} \right) \\
\geq Q \left(D(\lambda_n)y\right)\left(1+ \frac {|\alpha_n|{\rm
Im} z}{2k}\right) \geq
Q(x)\left(1+ \frac {|\alpha_n|{\rm Im} z }{2k}\right) \\
> Q(x)\left( 1 + \frac{|\alpha_n|{\rm Im} z}{2k(1+
|z|)} \right).
\end{multline*}

\end{proof}

Returning to the proof of the theorem, we have
\begin{equation*}\aligned
P_n(z)= H(s_n) & D(\lambda_n)R(\alpha_n) H(z)\cdot\,...\,\cdot
H(s_1)D(\lambda_1)R(\alpha_1)H(z)= \\
& \prod_{1 \leq j \leq n}^{\curvearrowleft}
H(s_j)D(\lambda_j)R(\alpha_j)H(z)
\endaligned
\end{equation*}
(recall that $\displaystyle\prod_{1 \leq j \leq
n}^{\curvearrowleft}A_j$ stands for the matrix product
$A_nA_{n-1}\ \ldots \ $).

Denote
\begin{equation*}
B_n(z)= H(z/2)P_n(z)H(z/2)^{-1}=\prod_{1 \leq j \leq
n}^{\curvearrowleft} H(z/2)H(s_j)D(\lambda_j)R(\alpha_j)H(z/2).
\end{equation*}
Then
$$
\overline{\lim_{n\rightarrow\infty}}\ \ \frac{\log
\|P_n(z)\|}{n}=\overline{\lim_{n\rightarrow\infty}}\ \ \frac{\log
\|B_n(z)\|}{n}.
$$

Let us consider the vector $\displaystyle x = \left(\frac{i}{\sqrt
2},\frac{1}{\sqrt 2}\right)$, $|x|^2=2Q (x) = 1$. Then, due to the
fact that for arbitrary $y$
$$
\|y\|^2= y_1^2 + y_2^2 \geq 2|y_1y_2|=2|y_1 \overline {y_2}| \geq
2{\rm Im} (y_1\overline y_2)= 2Q(y),
$$
we obtain
\begin{equation*}\aligned
& \log \|B_n (z)\|\geq \frac{1}{2}\log \|B_n(z)x\|^2\geq
\frac{1}{2}
\log \left(2Q\left(B_n(z)x \right) \right)\\
& \geq\frac{1}{2} \sum_{j=1}^n\log \left(1 + \frac{|\alpha_j|{\rm
Im} \frac{z}{2}}{2k(1+ |\frac{z}{2}|)}\right)\geq
\frac{1}{4}\sum_{j=1}^n \frac{|\alpha_j|{\rm Im}
\frac{z}{2}}{2k(1+ |\frac{z}{2}|)} = \frac{{\rm Im}
\frac{z}{2}}{8k(1+ |\frac{z}{2}|)}\sum_{j=1}^n |\alpha_j|
\endaligned
\end{equation*}
where we used that $\displaystyle\log(1+x)\geq \frac{1}{2}x$ for $x \in
(0;1)$ and $\displaystyle \frac{|\alpha_j|{\rm Im} z}{2k(1+ |z|)}
\leq \frac{|\alpha_j|}{2k}< 1$.

As a consequence, we obtain:
\begin{equation*}
\overline{\lim_{n\rightarrow\infty}}\ \ \frac{\log
\|P_n(z)\|}{n}\geq \overline{\lim_{n\rightarrow\infty}}\ \
\frac{{\rm Im} \frac{z}{2}}{8k(1+ |\frac{z}{2}|)}\cdot
\frac{1}{n}\sum_{j=1}^n |\alpha_j|=\frac{q \cdot {\rm Im}
\frac{z}{2}}{8k(1+ |\frac{z}{2}|)}>0.
\end{equation*}

This proves (\ref{complex}), that is the exponential growth of
$\|P_n(z)\|$ for non-real $z$. Thus $E\subset\mathbb{R}$.

The claim that $\left|E\right|<\infty$ follows now from
Lemma~\ref{harmon} applied to the  subharmonic functions
$$
u_n(z)=\frac{\log \|P_n(z)\|}{n}.
$$
Indeed, the norm of the matrix $H(z)=\begin{pmatrix} 1 & z \\
0 & 1\end{pmatrix}$ does not exceed $1+|z|$. Hence
\begin{equation*}
\|P_n(z)\|= \Big\|\prod_{1 \leq j \leq n}^{\curvearrowleft}
\left[\Phi_jH(z)\right]\Big\|\leq \|H(z)\|^n\cdot\prod_{1 \leq j
\leq n }^{\curvearrowleft} \| \Phi_j\| \leq
\left(1+|z|\right)^n\cdot k^n
\end{equation*}
Therefore
\begin{equation*}
\frac{1}{n}\log\|P_n(z)\|\leq\log\left(1+|z|\right)+ \log k.
\end{equation*}
Thus the functions $u_n$ satisfy the majorization condition of
Lemma \ref{harmon}. By definition, $\limsup u_n(z) = 0$ for $z\in
E$. If $|E| = \infty$ then Lemma \ref{harmon} implies that
$\limsup u_n(z) \le 0$ for all $z\in \mathbb{C}$, however, this
contradicts (\ref{complex}).
\end{proof}

\section{Constructing an exceptional set containing a given sequence}

Let, as above, $\displaystyle H(t)=\begin{pmatrix} 1 & t \\
0 & 1 \end{pmatrix}$, $\Phi_*$ be a sequence of matrices from
$SL(2,\mathbb{R})$. The exceptional set was defined as follows
$$
\mathfrak B(\Phi_*)=\{t\geq 0: \sup_K \Big \|\prod_{1 \leq k \leq
K}^{\curvearrowleft}\Phi_kH(t)\Big\| <\infty \}
$$
We have proved that this set always has finite measure.
Nevertheless it can be unbounded. Moreover, it can contain an
arbitrary given sequence:
\begin{theorem}\label{seq}
For every sequence $\{t_n\}$ of positive numbers there exists a
sequence $\Phi_*$ such that $\{t_n\} \subset \mathfrak B(\Phi_*)$.
\end{theorem}

\begin{proof} First let us note the following fact: for every
$SL(2,\mathbb{R})$-matrix $\displaystyle A = \begin{pmatrix} a & b \\
c & d \end{pmatrix}$ there exists an orthogonal matrix
$\displaystyle R=\begin{pmatrix}
\cos \alpha & -\sin \alpha \\ \sin \alpha & \cos \alpha \end{pmatrix}$
for which $(RA)^2=-\mathbf{1}$. To prove it is sufficient to choose
$R$ such that ${\rm tr}\ (RA)=0$, for example to take
$\displaystyle\alpha={\rm arctg} \frac{c-b}{a+d}$.

Now let us construct two sequences $\{A_n(t)\}$ and $\{R_n\}$ in
the following way. For $A_1(t)=H(t)$ we will choose $R_1$ such
that $(R_1A_1(t_1))^2=-\mathbf{1}$. Further for each $n \in
\mathbb{N}$ we define $A_n(t)=(R_{n-1}A_{n-1}(t))^2$ and we choose
$R_n$ such that $(R_nA_n(t_n))^2=-\mathbf{1}$.

Now we define sequence $\Phi_*$ as follows: $\Phi_n = R_{k+1}\
\ldots \ R_2R_1$ where $k$ is the the largest $j$ such that $2^j$
divides $n$. Thus we have: $\Phi_1 = R_1$, $\Phi_2 = R_2R_1$,
$\Phi_3 = R_1$, $\Phi_4 = R_3R_2R_1$, $\Phi_5 = R_1$, $\Phi_6 =
R_2R_1$, $\Phi_7 = R_1$, $\Phi_8 =R_4 R_3R_2R_1$, $\Phi_9 = R_1$,
...

Then the evolution sequence  $ P_n(t) = \Phi_n H(t)\Phi_{n-1} H(t)
\, ... \, \Phi_1 H(t)$  has the form
\begin{equation*}
\ldots \
R_3R_2R_1H(t)R_1H(t)R_2R_1H(t)R_1H(t)R_3R_2R_1H(t)R_1H(t)R_2R_1H(t)R_1H(t)
\end{equation*}
\begin{equation*}
= \ldots \ R_2 A_2(t)R_4R_3R_2A_2(t)R_2A_2(t)R_3R_2A_2(t)R_2A_2(t)
\end{equation*}
\begin{equation*}
= \ldots \ R_4R_3A_3(t)R_3A_3(t)R_4R_3A_3(t)R_3A_3(t)
\end{equation*}
and so on. Thus, for any $k$,
\begin{equation*}
P_n(t) = B(n,k)\ldots \
R_{k+2}R_{k+1}R_kA_k(t)R_kA_k(t)R_{k+1}R_kA_k(t)R_kA_k(t)
\end{equation*}
where the factor $B(n,k)$ is a product of not more than $N = N(k)$
matrices which are either orthogonal or equal to $H(t)$. It
follows that the norm of $B(n,k)$  does not exceed a constant
depending only on $k$ and $t$:
$$
\|B(n,k)\|\le C(k,t).
$$
Since
$(R_kA_k(t))^2 = -\mathbf{1}$ for $t=t_k$,  one has
$$
\|P_n(t_k)\|\le C(k,t_k).
$$
This means that $t_k\in \mathfrak B(\Phi_*)$.
\end{proof}

\section{Constructing an essentially unbounded exceptional set}

In this section we will construct an example of a sequence
$\Phi_*\subset SL(2,\mathbb{R})$ for which the exceptional set
$$
\mathfrak B(\Phi_*)=\{t\geq 0: \sup_K \Big \|\prod_{1 \leq k \leq
K}^{\curvearrowleft}\Phi_kH(t)\Big\| <\infty \}
$$
is essentially unbounded, that is $\displaystyle|\mathfrak
B(\Phi_*)\bigcap [a,+\infty)| > 0$ for all $a>0$ .

Let us consider a sequence of matrices $\displaystyle
M_j=M(c_j)=\begin{pmatrix}
1 & 0 \\ c_j & 1 \end{pmatrix}$ $(j\geq 0)$ with $c_j \neq 0$.
We define the sequence $\{\Phi_k\}$ $(k\geq1)$ in the following
way: $\displaystyle\Phi_k=M(c_{j(k)})$ where $j(k)$ is the largest
$j$ such that $2^j$ divides $k$. The first few terms of the
sequence $\Phi_*$ are
$$M_0, M_1, M_0, M_2, M_0, M_1, M_0, M_3, M_0, M_1, M_0,
M_2, M_0, M_1, M_0,...
$$
("the abacaba order" \cite{abc}).

\begin{theorem}\label{EUS}
There exists a sequence $\{c_j\}$ such that the set $\mathfrak
B(\Phi_*)$ is essentially unbounded.
\end{theorem}

The proof of this statement will be given in the next section.
Here we only outline its basic ideas.

First of all note that our choice of the sequence $\Phi_k$ implies
that the partial products $\displaystyle \prod_{1 \leq k \leq
K}^{\curvearrowleft}\Phi_kH(t)$ with diadic numbers ($K = 2^m$)
are related by a simple recurrent formula. Namely let us define a
sequence of matrix-functions $A_n(t)$, $n\ge -1$, as follows:
\begin{equation*}
A_{-1}(t) = H(t), \;\;A_{n+1}(t)=A_n(t)M(c_{n+1})A_n(t).
\end{equation*}
Then it is easy to check that $\displaystyle \prod_{1 \leq k \leq
2^m}^{\curvearrowleft}\Phi_kH(t) = M(c_m)A_{m-1}(t)$. More
generally, $\displaystyle \prod_{2^l+1 \leq k \leq
2^{l+m}}^{\curvearrowleft}\Phi_kH(t) = M(c_m)A_{l-1}(t)$. So it is
possible to express all partial products via $A_n(t)$. For
example, for $K=84$, we have $\displaystyle84=2^6+2^4+2^2$ and,
respectively,

\[\prod_{1 \leq k \leq
84}^{\curvearrowleft}\Phi_kH(t)=M(c_2)A_1(t)M(c_4)A_3(t)M(c_6)A_5(t).
\]

The general formula is
\begin{equation}\label{gendec}
\prod_{1 \leq k \leq K}^{\curvearrowleft} \Phi_kH(t) = \prod_{1
\leq \ell \leq L}^{\curvearrowleft} M_{j_\ell}A_{j_\ell-1}(t)
\end{equation}
(here $\{j_\ell\}$ is the strictly increasing finite sequence of
integers such that $\displaystyle
K=2^{j_1}+2^{j_2}+\ldots+2^{j_L}$). It follows from (\ref{gendec})
that for proving the boundedness of the sequence of all partial
products for a given $t$, it will be sufficient to find an upper
bound for the norms of partial products
$\displaystyle \prod_{1 \leq k \leq
K}^{\curvearrowleft}\Phi_kH(t)$ for all $K$ that are multiples of
$2^m$ for some integer
$m$. Note that only $A_n$ with $n \geq m$ can appear in such partial
products.

Suppose that for some $t>0$ and for some integer $m$ the following
condition holds:
\begin{equation}\label{trace}
-2 < {\rm tr}(A_n(t)) < 2, {\rm\;\; for\;\; all\;\; }n \ge m.
\end{equation}
Then the eigenvalues of $A_n(t)$ are complex conjugate and the
matrices are similar to diagonal ones:

$\displaystyle
A_n(t)=S_n(t)D_n(t)S_n(t)^{-1}$ where $ \displaystyle D_n(t)=
\begin{pmatrix} \lambda_n(t) & 0 \\
0  & \overline{\lambda_n(t)}\end{pmatrix}$.

Suppose also that, for all $n \ge m$,

\begin{equation}\label{close}
\|S_{n+1}(t)-S_n(t)\|<
\varepsilon_n
\end{equation}
where the numbers $\varepsilon_n$ are such that
$\sum_{n=m}^{\infty} \varepsilon_n < \infty$.

Under these conditions it can be proved, using (\ref{gendec}),
that all products
$\displaystyle\prod_{1 \leq k \leq
K}^{\curvearrowleft}\Phi_kH(t)$ with $\displaystyle K \in
2^{m+1}\mathbb{Z}$ are bounded by the same constant.

Thus, it remains to construct an essentially unbounded set $E$
such that the conditions (\ref{trace}) and (\ref{close}) are
satisfied for all $t\in E$ (with $m$ depending on $t$).  We will
define $E$ as $\displaystyle \bigcup_{k=0}^{\infty}\bigcap_{n=k}^
{\infty}E_n$ where sets $E_n$ are constructed inductively.

The possibility of the induction steps is provided by two
auxiliary results (Lemmas 7 and 8) which state that (under some
conditions) a set $F\subset \mathbb{R}$ on which the condition
(\ref{trace}) holds, can be slightly reduced and, respectively,
can be extended by adding an interval located arbitrarily far away
from the origin in such a way that on the new set the inequalities
of type (\ref{close}) hold.

For the beginning of the induction process we take a closed
interval $E_0 \subset (0, +\infty)$ such that ${\rm tr}\
A_0(t)\in(-2,2)$ on $E_0$, and choose a sequence $\displaystyle
\{\varepsilon_n\}$ with $\sum_{n=1}^{\infty}\varepsilon_n<
\frac{1}{3}|E_0|$.

Then we choose (using Lemma 7) a closed subset
$\displaystyle\widetilde E_0 \subset E_0$ such
that $\displaystyle|E_0 \setminus \widetilde E_0|<\varepsilon_1 |E_0|$
and the conditions $\displaystyle{\rm tr}\  A_1(t)\in (-2,2)$ and
$\displaystyle\|S_1(t)-S_0(t)\|<\varepsilon_1$ hold on
$\displaystyle\widetilde E_0$.

Now using Lemma 8 we find a closed interval $I_1$ such that its
left endpoint is larger than $\sup E_0$ and ${\rm tr}\ A_1(t) \in
(-2,2)$ on $I_1$. We put $\displaystyle E_1=\widetilde E_0 \cup
I_1$.

In this manner we proceed with constructing the intervals $I_n$
and sets $E_n$. Namely, on the $n$-th step we get
a set $E_n$ such that $I_n \subset E_n$, and choose its subset
$\displaystyle\widetilde E_n$, satisfying $\displaystyle|E_n \setminus
\widetilde E_n|<\varepsilon_n|I_n|$, in such a way that
the conditions $\displaystyle{\rm tr}\  A_j(t)\in (-2,2)$ and
$\displaystyle\|S_j(t)-S_{j-1}(t)\|<\varepsilon_n$ are fulfilled
for $t \in \widetilde E_n $ and for all $j \leq n$. Then we set
$E_{n+1} = E_n\cup I_{n+1}$ where the left endpoint of $I_{n+1}$
is larger than $\sup E_n$.

The smallness of the deleted parts of the sets $E_n$ and the
condition $I_n \subset (n,\infty)$ provide the essential
unboundedness of the set $E$.

\section{Proof of Theorem \ref{EUS}}

We start with two auxiliary results.

Let us call a real polynomial matrix function $A(t)=
\begin{pmatrix} a_{11}(t) & a_{12}(t)\\ a_{21}(t) &
a_{22}(t)\end{pmatrix}$ and a compact set $E\subset (0,+\infty)$ a
good pair if the following conditions hold:
\begin{itemize}
\item[(i)] $\det A(t)=1$ for all $t$;
\item[(ii)] ${\rm tr}\  A(t)=a_{11}(t)+a_{22}(t)$ is a non-constant
polynomial;
\item[(iii)] ${\rm tr}\  A(t)\in (-2,2)$ for all $t\in E$.
\end{itemize}

If $(A(t), E)$ is a good pair, then, according to the spectral
theorem, one can find continuous functions
$\lambda:E\to\mathbb{T}$, ${\rm Im}\ \lambda \ne 0$ and $S:E\to
SL(2,\mathbb{C})$ such that $\displaystyle A(t)=S(t)D(t)S(t)^{-1}$
where $ D(t)= \begin{pmatrix} \lambda(t) & 0 \\ 0  &
\overline{\lambda(t)}
\end{pmatrix}$.

Choosing $c\in \mathbb{R}$, we set  $\widetilde
A(t):=A(t)M(c)A(t)$.

\begin{lemma}\label{mainl}
Assume that $(A(t),E)$ is a good pair. Let $\varepsilon>0$. Then
there exists $\delta>0$ such that, for every real $c$ with
$|c|<\delta$, there exists a compact set $\widetilde E \subset E$
such that

{\rm (a)} $\;\;|E\setminus\widetilde E|<\varepsilon$;

{\rm(b)} $\;\;$ the pair $(\widetilde A(t), \widetilde E)$ is
good;

\noindent and

{\rm (c)} $\;\;$ A matrix-function $\widetilde{S}(t)$,
diagonalizing $\widetilde A(t)$: $$\widetilde{A}(t) =
\widetilde{S}(t)\widetilde{D}(t)\widetilde{S}(t)^{-1}$$ can be
chosen in such a way that  $\|\widetilde S(t)-S(t)\|<\varepsilon$.
\end{lemma}

\begin{proof}[Proof.]
Note that if $|c|$ is small then the function $\widetilde A(t)$ is
close to $A(t)^2$ on $E$:
$$
\|\widetilde A(t)-A(t)^2\|=\|A(t)[M(c)-I]A(t)\|\le
|c|\cdot\|A(t)\|^2\le C_1 \delta)
$$
It follows that it is ``almost diagonalized'' by means of $S(t)$
$$
\|S(t)^{-1}\widetilde A(t) S(t)-D(t)^2\|\le C_1\delta\cdot
\|S(t)\|^2 \le C_2 \delta
$$
on $E$. Since the matrix function $B(t) = S(t)^{-1}\widetilde A(t)
S(t)$ is similar to $\widetilde A(t)$ (so has the same diagonal
part $\widetilde D(t)$), the pair $(\widetilde A(t), \widetilde
E)$ (whatever $\widetilde{E}$ be chosen) is good if and only if
$(B(t), \widetilde E)$ is good. So we will deal with $B(t)$. Let
us first of all show that $B(t)$ can be diagonalized
$$
B(t) = V(t)\widetilde D(t) V(t)^{-1}
$$
via a matrix $V(t)$ which is close to $I$. It will follow that
\linebreak $\displaystyle\widetilde A(t) = S(t)V(t)\widetilde D(t)
V(t)^{-1}S(t)^{-1}$, and $\widetilde S(t) = S(t)V(t)$ is close to
$S(t)$.

We already have that
$$
\|B(t) - D(t)^2\|\le C_2\delta
$$
on $E$.
Now, $\displaystyle D(t)^2= \begin{pmatrix} \lambda^2(t) & 0 \\
0  & \overline{\lambda^2(t)}
\end{pmatrix}$
is a continuous diagonal matrix-function with distinct diagonal
elements for all $t\in E$ except, maybe, finitely many $t$
satisfying the equation ${\rm tr}\ A(t)=0$ (in which case
$\lambda(t)=\pm i$ and $\lambda^2(t)=\overline{\lambda^2(t)}=-1$).
Let $G$ be any open set containing those exceptional $t$ and such
that $|G|<\varepsilon$. Put $\widetilde E=E\setminus G$. Then
${\rm Tr}\ [D(t)^2]=2{\rm Re}\ [\lambda^2(t)]\subset [-a,a]$ for
all $t\in \widetilde E$ with some $a<2$. Let $\delta>0$ be so
small that $2C_2\delta <2-a$. Then ${\rm tr}\ B(t) \in(-2,2)$ and,
therefore, the eigenvalues of $B(t)$ are $\widetilde\lambda(t)$
and $\overline{\widetilde\lambda(t)}$ with
$|\widetilde\lambda(t)|=1$. Moreover, $\widetilde\lambda(t)$ is a
continuous function of $t$ and $\displaystyle
|\widetilde\lambda(t)-\lambda^2(t)|\le C_3\sqrt\delta$. Let now
$\displaystyle m=\min_{t\in \widetilde E}|{\rm Im}\
[\lambda^2(t)]|$ and note that $m>0$. Also, let
$$
B(t)-D(t)^2=:\Delta(t)=
\begin{pmatrix} \Delta_{11}(t) & \Delta_{12}(t)\\
\Delta_{21}(t) & \Delta_{22}(t)
\end{pmatrix}.
$$
Then the matrix $V(t)$ whose columns are eigenvectors of $B(t)$ is
$$
V(t)=\begin{pmatrix}
\overline{\lambda^2(t)}-\widetilde\lambda(t)+\Delta_{22}(t) &
\Delta_{12}(t)
\\
-\Delta_{21}(t) &
\overline{\widetilde\lambda(t)}-\lambda^2(t)-\Delta_{11}(t)
\end{pmatrix}\,.$$
The exact formula for $V(t)$ doesn't matter but it is important
that $V(t)$ can be chosen to be a continuous function of $t$ that
is close to a diagonal matrix with equal non-zero elements on the
diagonal when the perturbation $\Delta(t)$ is close to $0$. Note
that $\det V(t)\ne 0$ if $\delta$ is small enough. Let now
$\displaystyle \widetilde V(t)=\frac 1{\sqrt{\det V(t)}}V(t)$
where the branch of the square root of the determinant is chosen
in such a way that it equals
$\overline{\lambda^2(t)}-\lambda^2(t)$ when $\Delta(t)=0$. Then
the norm $\|\widetilde V(t)-I\|$ can be made arbitrarily small if
$\delta$ is small enough. It remains to note that $B(t)=
\widetilde V(t)\widetilde D(t)\widetilde V(t)^{-1}$ where
$\displaystyle \widetilde D(t)= \begin{pmatrix}
\widetilde\lambda(t) & 0 \\
0 & \overline{\widetilde\lambda(t)}
\end{pmatrix}$,
so we can put $\widetilde S(t)=S(t)\widetilde V(t)$.

We proved the  statement (c) of the lemma. The statement (a)
follows from the inequality $|G| < \varepsilon$. To have (b) we
must check conditions $(i-iii)$ for the matrix $\widetilde A(t)$.
The condition $(i)$ is obvious because the product of three
matrices of determinant $1$ is still a matrix of determinant $1$.
The condition $(iii)$ is proved above: we have shown that ${\rm
tr}\ B(t)\subset (-2,2)$ but ${\rm tr}\ \widetilde A(t) = {\rm
tr}\ B(t)$. It remains to check $(ii)$. A direct computation
yields
\begin{multline*}
{\rm tr}\ \widetilde
A(t)=a_{11}^2(t)+2a_{12}(t)a_{21}(t)+a_{22}^2(t)+
ca_{12}(t)[a_{11}(t)+a_{22}(t)]
\\
=[{\rm tr}\ A(t)]\cdot[{\rm tr }\ A(t)+ca_{12}(t)]-2\,.
\end{multline*}
Since ${\rm tr}\ A(t)$ is a non-constant polynomial, the whole
expression is a non-constant polynomial for all sufficiently small
$c$ and we are done.
\end{proof}

Note that by the construction, the matrix $V(t)$ is unimodular.
Hence $\widetilde{S}(t)$ is unimodular if $S(t)$ is such. This
shows that in further constructions, based on Lemma \ref{mainl},
we may assume that the obtained matrix functions $S_n(t)$ are
unimodular.

In the following lemma, which can be regarded as a modification of
Lemma \ref{mainl}, we preserve the notations $\widetilde{E}(t)$
and $\widetilde{S}(t)$. For brevity, let us call a polynomial
matrix function $\displaystyle P(t)= (p_{ij}(t))$ {\it upper right
dominating} if the degree of the polynomial $p_{12}(t)$ is more
than the degrees of others $p_{ij}(t)$.

\begin{lemma}\label{modif}
Assume that $(A(t),E)$ is a good pair and that the polynomial
matrix $A(t)^2$ is upper right dominating. Let $\varepsilon>0$ and
$N> 0$ be given.  Then there exists $\delta > 0$ and a compact interval
$I\subset (N,\infty)$ such that, for every real $c$ with $0<|c|\le \delta$,
there exists a compact set $\widetilde E\subset E$ such that
$|E\setminus\widetilde E|<\varepsilon$; $(\widetilde A(t),
\widetilde E\cup I)$ is a good pair; and, moreover, $\|\widetilde
S(t)-S(t)\|<\varepsilon$ on $\widetilde E$.
\end{lemma}

\begin{proof}[Proof.]
By the proof of Lemma \ref{mainl}, we get $\delta_1$ such that for
$|c| < \delta_1$ one can find $\widetilde E\subset E$ satisfying
the conditions: $|E\setminus\widetilde E|<\varepsilon$,
$(\widetilde A(t), \widetilde E)$ is a good pair and $\|\widetilde
S(t)-S(t)\|<\varepsilon$ on $\widetilde E$.

Let $A(t)^2=(b_{ij}(t))$. Since $\widetilde A(t)=A(t)M(c)A(t)$,
$$
{\rm tr}\  \widetilde A (t)={\rm tr}\ [M(c)A(t)^2]= cb_{12}(t) +
q(t),
$$
where the degree of $q(t)$ is less than the degree of $b_{12}(t)$.
Hence if $|c|$ is less than some $\delta_2$ and has the
appropriate sign, then there is $t_0
> N$ for which ${\rm tr }\ \widetilde A (t_0) = 0$. So the condition
(ii) holds on some interval $I$ around $t_0$. This shows that
$(\widetilde A(t), \widetilde E\cup I)$ is a good pair.

It remains to set $\delta = \min\{\delta_1,\delta_2\}$.
\end{proof}

The number $\delta$, constructed in Lemma \ref{modif}, will be
denoted by $\delta(A(t),E,\varepsilon,N)$. To underline that in
the construction of the interval $I$ and the set $\Omega =
\widetilde E \cup I$ the number $c$ from $(-\delta,0)$ or
$(0,\delta)$ is used, we will denote them by $I =
I(A(t),E,\varepsilon,N,c)$ and $\Omega(A(t),E,\varepsilon,N,c)$
respectively.

Now we can prove the theorem.

\begin{proof}[{\bf Proof of Theorem \ref{EUS}}]

We shall start with the matrix $A_0(t)=H(t)M_0H(t)$ and note that
if $c_0<0$, then there exists a closed interval $E_0\subset
(0,+\infty)$ of positive length such that $(A_0(t), E_0)$ is a
good pair. Choose $\varepsilon_0 = |E_0|/3$.

We will construct the sequences of numbers $c_j$,
$\varepsilon_j$, matrix functions $A_j(t)$, compact sets $E_j$ and
compact intervals $I_j$ inductively.

Suppose that these sequences are constructed for $j < n$. Then
set
$$
\varepsilon_n = \frac{1}{3^n}\min_{j<n}\{|I_j|\}, \qquad
\delta = \delta(A_{n-1},E_{n-1},\varepsilon_n,n)
$$
and choose $c_n$ with $|c_n| < \delta$ and with appropriate sign.
Let
$$
I =I(A_{n-1},E_{n-1},\varepsilon_n,n,c_n), \qquad
E_n = \Omega(A_{n-1},E_{n-1},\varepsilon_n,n,c_n).
$$

For these definitions be correct, we have to check that the
pairs $(A_n(t),E_n)$ are good and that $A_n(t)^2$ are right upper
dominating.

The first property follows by induction from Lemma \ref{modif}. To
prove the \linebreak second one, note that for each $n$, the
function $A_n(t)^2$ is a product \linebreak $HMHMH\ldots MH$ of
matrices in which each $H$ is either $H(t)$ or $H(2t)$ and each
$M$ is $M(c)$ with some $c\ne 0$ (possibly different for different
$M$'s). Let $p$ be the number of $M$'s in the product. Then
$A_n(t)^2= (b_{ij}(t))$ where $b_{ij}$ are polynomials with
degrees of $b_{11}(t)$ and $b_{22}(t)$ equal $p$, degrees of
$b_{12}$ and $b_{21}$  equal $(p+1)$ and $(p-1)$ respectively. The
correctness is proved.

Let us set
\begin{equation*}
E =\bigcup_{k=0}^{\infty}\bigcap_{n=k}^{\infty}E_n.
\end{equation*}
It
follows from the choice of the numbers $\varepsilon_j$ that
$|E\cap I_n| \neq 0$ for each n. Hence $E$ is essentially
unbounded. We have to prove that for each $t\in E$ the sequence
$\displaystyle \Big \|\prod_{1 \leq k \leq K}^{\curvearrowleft}
\Phi_kH(t)\Big\|$ is bounded.

If $t\in E$ then there is $m$ such that $t\in E_n$ for each $n \ge
m$. Then for any $K$ which is divided by $2^{m+1}$, the partial
product $\displaystyle \prod_{1 \leq k \leq K}^{\curvearrowleft}
\Phi_kH(t)$ can be written (see (\ref{gendec}) as
\begin{equation*}
\prod_{1 \leq k \leq K}^{\curvearrowleft} \Phi_kH(t) = \prod_{1
\leq \ell \leq L}^{\curvearrowleft} M_{j_\ell}A_{j_\ell-1}(t)=
\prod_{1 \leq \ell \leq L}^{\curvearrowleft}
[M_{j_\ell}S_{j_\ell-1}(t) D_{j_\ell-1}(t)S_{j_\ell-1}(t)^{-1}]
\end{equation*}
where $\{j_\ell\}$ is the strictly increasing finite sequence of
integers such that $\displaystyle
K=2^{j_1}+2^{j_2}+\ldots+2^{j_L}$ and all $j_\ell > m$. To
estimate the norm of this product, note that it consists of
several diagonal matrices of norm $1$, the matrix
$M_{j_1}S_{j_1-1}(t)$ in the beginning, the matrix
$S_{j_L-1}(t)^{-1}$ in the end and several matrices of the kind
$S_{j_\ell-1}(t)^{-1}M_{j_{\ell+1}}S_{j_{\ell+1}-1}(t)$ in the
middle. Now, $\|M_j\|$ are bounded. Also
\begin{multline*}
\|S_j(t)\|\le \|S_j(t)-S_{j-1}(t)\|+\ldots+\|S_{m+1}(t)-S_m(t)\|+\\
\|S_m(t)\|\le \|S_m(t)\|+\sum_{j\ge m+1}\varepsilon_j
\end{multline*}
are bounded for each such $t$. Since the matrices $S_j(t)$ are
unimodular, their inverse are also bounded:
\begin{equation*}
\|S_j(t)\| \le C(t), \;\;, \|S_j(t)^{-1}\| < C(t).
\end{equation*}
It remains to estimate the norms of
$\displaystyle S_{j_\ell-1}(t)^{-1}M_{j_{\ell+1}}S_{j_{\ell+1}-1}(t)$.
We have
\begin{multline*}
\|S_{j_\ell-1}(t)^{-1}M_{j_{\ell+1}}S_{j_{\ell+1}-1}(t)-I\|=\\
\|S_{j_\ell-1}(t)^{-1} \Bigl( \left(M_{j_{\ell+1}}-I
\right)S_{j_{\ell+1}-1}(t)\Bigr. + \Bigl.
\left(S_{j_{\ell+1}-1}(t)- S_{j_\ell-1}(t)\right) \Bigr) \|\le
\\
\|S_{j_\ell-1}(t)^{-1}\|\cdot\Bigl(\|M_{j_{\ell+1}}-I\|\cdot\|S_{j_{\ell+1}-1}(t)\|\Bigr.
+
\Bigl. \|S_{j_{\ell+1}-1}(t)-S_{j_\ell-1}(t)\|\Bigr)\\
\le C(t)\left(|c_{j_{\ell+1}}|
C(t)+\sum_{j=j_\ell}^{j_{\ell+1}-1}\varepsilon_j\right)\,.
\end{multline*}

Hence
$$
\|S_{j_\ell-1}(t)^{-1}M_{j_{\ell+1}}S_{j_{\ell+1}-1}(t)\|\le
\exp\Bigl\{
C(t)\Bigl[C(t)|c_{j_{\ell+1}}|+\sum_{j=j_\ell}^{j_{\ell+1}-1}\varepsilon_j\Bigr]
\Bigr\}\,.
$$
Multiplying all the above estimates, we see that, for
$\displaystyle K  \in 2^{m+1}\mathbb{Z}$, one has
$$
\Bigl\|\prod_{1 \leq k \leq K}^{\curvearrowleft}
\Phi_kH(t)\Bigr\|\le C^2(t) \exp\Bigl\{ C(t)\Bigl[C(t)\sum_{j\ge
1}|c_{j}|+\sum_{j\ge 1}\varepsilon_j\Bigr] \Bigr\}\,.
$$
Therefore the partial products corresponding to $\displaystyle
K\in 2^{m+1}\mathbb{Z}$ are bounded for each $\displaystyle
t\in\bigcap_{n\geq m}E_n$. The products
corresponding to other $K$ differ from the products corresponding to
$\displaystyle K \in 2^{m+1}\mathbb{Z}$ by just N couples of
(uniformly) bounded matrices
($\displaystyle N \in \{1,\ 2,\ , ...,\ 2^{m+1}-1\}$). Therefore, all
the sequence
of partial products is  bounded for such $t$.

\end{proof}

\vspace{0.5cm} \noindent {\large Department of Mathematics,
Michigan State University, East Lansing, MI 48824, USA}

\noindent {\large \textit {E-mail}}: \hspace{0.5cm}
fedja@math.msu.edu

\vspace{0.5cm} \noindent {\large School of Mathematics, Tel Aviv
University, Tel Aviv 69778, Israel}

\noindent {\large \it {E-mail}}:\hspace{0.5cm} shulmank@yahoo.com

\end{document}